\documentclass[12pt,reqno]{amsart}
\usepackage[letterpaper,margin=1in,footskip=0.25in]{geometry}
\usepackage{amssymb}
\usepackage[only,llbracket,rrbracket]{stmaryrd}
\usepackage{microtype}
\usepackage{mathtools}
\usepackage{enumitem}
\usepackage[noadjust]{cite}

\PassOptionsToPackage{dvipsnames}{xcolor}
\usepackage{tikz-cd}

\PassOptionsToPackage{pdfusetitle,colorlinks,pagebackref,linktocpage,bookmarksdepth=3}{hyperref}
\usepackage{bookmark}
\hypersetup{citecolor=OliveGreen,linkcolor=Mahogany,urlcolor=Plum}
\usepackage[hyphenbreaks]{breakurl}

\newtheorem{theorem}{Theorem}[section]
\newtheorem{corollary}[theorem]{Corollary}
\newtheorem{lemma}[theorem]{Lemma}
\newtheorem{proposition}[theorem]{Proposition}

\newtheorem{innercustomthm}{Theorem}
\newenvironment{customthm}[1]
  {\renewcommand\theinnercustomthm{#1}\innercustomthm}
  {\endinnercustomthm}

\theoremstyle{definition}
\newtheorem{definition}[theorem]{Definition}

\theoremstyle{remark}
\newtheorem{remark}[theorem]{Remark}

\theoremstyle{question}
\newtheorem{question}[theorem]{Question}

\newtheoremstyle{cited}{.5\baselineskip\@plus.2\baselineskip\@minus.2\baselineskip}{.5\baselineskip\@plus.2\baselineskip\@minus.2\baselineskip}{\itshape}{}{\bfseries}{\bfseries .}{5pt plus 1pt minus 1pt}{\thmname{#1}\thmnumber{ #2}\thmnote{ \normalfont#3}}
\theoremstyle{cited}
\newtheorem{citedthm}[theorem]{Theorem}

\newtheoremstyle{citeddef}{.5\baselineskip\@plus.2\baselineskip\@minus.2\baselineskip}{.5\baselineskip\@plus.2\baselineskip\@minus.2\baselineskip}{}{}{\bfseries}{\bfseries .}{5pt plus 1pt minus 1pt}{\thmname{#1}\thmnumber{ #2}\thmnote{ \normalfont#3}}
\theoremstyle{citeddef}
\newtheorem{citeddef}[theorem]{Definition}

\makeatletter
\def\l@subsection{\@tocline{2}{0pt}{2pc}{6pc}{}}
\makeatother

\begin{document}
\title{Coherence of absolute integral closures}

\author{Shravan Patankar}
\address{Department of Mathematics, Statistics and Computer Science\\University
of Illinois at Chicago\\Chicago, IL 60607-7045\\USA}
\email{\href{mailto:spatan5@uic.edu}{spatan5@uic.edu}}
\urladdr{\url{https://spatan5.people.uic.edu/}}

\makeatletter
  \hypersetup{
    pdfauthor={Shravan Patankar},
    pdfsubject=\@subjclass,pdfkeywords={Coherence, absolute integral closures}
  }
\makeatother

\begin{abstract}
We prove that the absolute integral closure $R^{+}$ of an equicharacteristic zero noetherian complete local domain $R$ is not coherent, provided $\dim(R)\geq 2$. As a corollary, we give an elementary proof of the mixed characteristic version of the result due to Asgharzadeh and extend it to dimension $3$. Furthermore, we apply the methods of Aberbach and Hochster used to prove the positive characteristic version of this result to study F-coherent rings and our work naturally suggests a mixed characteristic analogue of a result of Smith.
\end{abstract}

\maketitle

\tableofcontents

\section{Introduction}
Throughout this article, all rings are assumed to be commutative and contain an identity element. Recall from \cite{Art71} that the \emph{absolute integral closure} of a domain $R$, denoted by $R^{+}$, is the integral closure of $R$ inside an algebraic closure of its fraction field. As Aberbach and Hochster remark in \cite[Introduction]{AH97}, although $R^{+}$ is never noetherian (if $\dim(R) \geq 1$), there is a tremendous motivation to study such algebras due to the deep and seminal results surrounding them: Hochster-Huneke \cite{HH94} and Bhatt's \cite{Bha20} results roughly state that $R^{+}$ is (cohomologically) Cohen-Macaulay in positive and mixed characteristics respectively and Smith's \cite{Smi94} work shows that $R^{+}$ has intimate connections with tight closure theory and rational singularities. We remark that although $R^{+}$ is never Cohen-Macaulay when $R$ has equicharacteristic zero if $\dim(R)\geq3$, there has been some speculation as to whether certain `almost' variants hold \cite{RSS07}. 
\par
\begin{citeddef}[{\citeleft\citen{Gla89}\citemid Def.\ on p.\
  44\citeright}]\label{def:Coh}
A ring $R$ is said to be \textbf{coherent} if every finitely generated ideal of $R$ is finitely presented.
\end{citeddef}

Clearly, noetherian rings are coherent and coherence is a natural condition to investigate in non-noetherian (for example, analytic) situations. Hence, it is natural to ask whether $R^{+}$ is coherent. We show the following:

\begin{customthm}{\ref{thm:Coho}}
Let $R$ be an equicharacteric zero (i.e. $\mathbb{Q} \subset R$) noetherian complete local domain such that $\dim(R) \geq 2$. Then $R^{+}$ is not coherent.
\end{customthm}

Our key ingredient for the proof comes from the work of Roberts, Singh, and Srinivas \cite{RSS07}. They show that there are elements of arbitrarily low degree annihilating the local cohomology of the cone over an elliptic curve in characteristic zero by constructing the annihilators explicitly in the absolute integral closure. It follows that the dagger closure of a certain ideal is non-trivial and a standard argument shows that this cannot happen if $R^{+}$ is coherent. As a corollary, we recover the mixed characteristic version of the above result due to Asgharzadeh \cite[Theorem 6.6]{Asg17} (who proved it when $\dim(R) \geq 4$) and extend it to $\dim(R) = 3$, answering his explicit question from the statement of \cite[Theorem 6.6]{Asg17}.
\begin{corollary} \label{thm:cormst}
Let $R$ be a noetherian complete local domain of mixed characteristic. If $\dim(R) \geq 3$ then $R^{+}$ is not coherent. 
\end{corollary}

In contrast to Asgharzadeh's proof which uses the direct summand theorem \cite{And18}, or even the methods of Aberbach and Hochster \cite{AH97} who prove the corresponding positive characteristic statement, our methods are \textit{elementary}. In particular, we do not use perfectoid geometry or any of the recent advances in mixed characteristic commutative algebra and algebraic geometry.\footnote{Shimomoto has given another proof of Asgharzadeh's theorem (in private communication, see proof of Theorem \ref{asgthm}) using the main theorems of \cite{Bha20}. Similar to Asgharzadeh's proof, it uses highly non-trivial results and is unable to achieve our $\dim(R)=3$ improvement.} Furthermore, Corollary \ref{thm:cormst} supplements the story of $R^{+}$ in mixed characteristic which has garnered a tremendous amount of interest in recent years due to \cite{Bha20} and its applications to mixed characteristic singularities and algebraic geometry \cite{BMP+20}, \cite{TY20}, \cite[Remark 4.11]{MS21}.

\par
The positive characteristic version of the above statements is due to Aberbach and Hochster \cite[Section 4]{AH97} who develop and use the famous homological properties of perfect rings (see also \cite[Section 3]{BS17}, \cite{BIM18}). Perhaps as interesting as our results are the reasons as to why our methods don't work in positive characteristic. While similar dagger closure computations work in positive characteristic \footnote{The computations are much easier since dagger closure agrees with tight closure for complete local domains by \cite{HH91}.}, we crucially use that normal equicharacteristic zero domains are splinters and hence certain elements in the dagger closure of an ideal are not in its plus closure. In positive characteristic, dagger closure agrees with plus closure in many situations and for our example (with the coefficient field of positive characteristic) by the deep result of \cite{Smi94}. To the best of the author's knowledge, the correct analogue of the main result of \cite{Smi94} in mixed characteristic is not known (\cite[Remark 5.5]{Bha20}, \cite[Remark 4.11]{MS21}). 

\par
The proof ideas of our main theorems do not show that the $p$-adic completion of $R^{+}$, $\widehat{R^{+}}$, is not coherent and we discuss progress towards it in the fourth section. Although we largely use the methods of Aberbach and Hochster, new ingredients from mixed characteristic such as the homological properties of perfectoid rings from \cite{BIM18} and the theory of coherent rings are required. Our refinement of their methods yields an application to `$F$-coherent' rings and connections to Heitmann's recent work \cite{Hei21}, homological properties of valuation rings. As mentioned above, our main theorems are intimately connected to Smith's result \cite{Smi94}. In the fifth section, we propose an analogue of Smith's result in mixed characteristic and demonstrate the desired containment for the running example in this note and what is often the first example in tight closure theory, the degree $3$ Fermat curve. The techniques used in these sections are much more sophisticated in comparison.



\section{Preliminaries}
In this section, we recall definitions and lemmas on valuations, closures, and coherent rings which we need to prove the main theorems.

\begin{citeddef} \label{defv}
  Let $A$ be a domain. A valuation on $A$ is a function $v$ is a from A
  to $\mathbb{Q} \cup \infty$ such that
  \begin{enumerate}
      \item $v(a) = \infty$ if and only if $a=0$
      \item $v(ab) = v(a) + v(b)$ for all $a,b \in A$
      \item $v(a+b) \geq \min(v(a), v(b))$ for all $a,b \in A$
  \end{enumerate}
\end{citeddef}
We need the following standard lemma regarding coherent rings and `almost zero' elements, a similar argument and discussion appears in \cite[Proposition 2.3]{AH97}, \cite[Remark 2.4]{AS12} and \cite[Discussion 3.7, Thm. 3.8]{Shi10}:

\begin{lemma}\label{almo}
Let $(R,m)$ be a noetherian local domain and let $J$ be a finitely generated proper ideal of $R^{+}$. Then $J$ cannot contain elements which are arbitrarily small with respect to a valuation on $R^{+}$ positive on $J$. 
\end{lemma}

\begin{proof}
Let $S$ be a finite extension of $R$ which contains generators of $J$. Since $J$ is a proper ideal, there exists $P$, a prime ideal of $S$, containing $J$ and the generators of $J$. There exists a DVR $(V, tV)$ such that $S \subseteq V \subseteq K$ where $K = Frac(S)$ and $P = S \cap tV$. Note that this implies that valuations of elements in $R$ but not in $m$ is $0$. We can extend the valuation to $R^{+}$ by the discussion below. This shows there exists a valuation positive on $J$. Let $\alpha$ be the minimum of the valuations of the elements generating $J$. Since $v(a) \geq \alpha$ for all $a$ \emph{generating} $J$, we have that $v(a) \geq \alpha$ for all elements of $J$ (by properties listed in Definition \ref{defv}) and hence we are done.
\\
The existence and extension to $R^{+}$ of valuation rings above is well known and is mentioned without a reference in \cite[Discussion 3.7]{Shi10}, \cite[page 237]{HH91}, and \cite[page 238]{RSS07}; we give a few more details at the suggestion of the referee. The DVR $V$ above exists by \cite[Theorem 6.4.3]{HS18}. We have $S \subseteq V \subseteq K$ and hence we have $S^{+} \subseteq V^{+} \subseteq K^{+}$. $V^{+}$ has a maximal ideal containing the maximal ideal of $V$ by the Cohen-Seidenberg theorems. $V^{+}$ need not be a valuation domain in general (it is if $V$ is henselian), however, it is a Pr\"{u}fer domain by \cite[Chapter VI, §8.6, Prop. 6]{Bou89} and hence localizing at the aforementioned maximal ideal gives the desired valuation. See also \cite[Page 28]{Ho07}.
\\

\end{proof}

We recall the definitions of dagger closure and plus closure:

\begin{citeddef}[{\citeleft\citen{HH91}\citeright}]
  Let $(R,m)$ be a complete local domain and let $v:R^{+} \xrightarrow{} \mathbb{Q} \cup \infty$ be a valuation positive on $m$. Let $I$ be an ideal of $R$. The dagger closure of $I$, $I^{\dagger}$ is the set of all elements $x \in R$ such that there are elements $u \in R^{+}$ of arbitrarily small valuation with $ux \in IR^{+}$
\end{citeddef}

Dagger closure was defined as an attempt to provide a characteristic free ideal closure operation with properties similar to tight closure. Hochster-Huneke \cite[Page 236]{HH91} remark that a priori, each valuation may give a different dagger closure, and show that this is not the case if $R$ has positive characteristic by proving that dagger closure agrees with tight closure (\cite[Theorem 3.1]{HH91}) for complete local domains (see \cite[Proposition 1.8]{BS12} where it is shown that dagger closure and tight closure agree for domains essentially of finite type over an excellent local ring). We do not know if dagger closure depends on the choice of the valuation if $R$ has equicharacteristic zero or mixed characteristic. Although tight closure has been studied extensively in positive characteristic, dagger closure continues to be mysterious\footnote{Some progress towards understanding dagger closure has been made by St\"{a}bler in his thesis work \cite{Sta10}.} and we refer the reader to \cite[Questions 4.1, 4.2]{RSS07} for some basic questions about dagger closure that remain unanswered. The results of this note can be viewed as an application of the theory of dagger closure.

\begin{definition}
  Let $R$ be a domain and let $I$ be an ideal of $R$. An element $a \in R$ is said to be in the plus closure of $I$, $I^{+}$, if $a \in IR^{+}$.
\end{definition}

\begin{definition}
  Let $R$ be a noetherian domain. $R$ is said to be a splinter if the inclusion map $R \rightarrow S$ splits (as a map of $R$-modules) for any module finite extension $S$ of $R$.
\end{definition}
The study of splinters has been a celebrated endeavor with a rich history. The direct summand conjecture, now a theorem, says precisely that regular rings are splinters (refer to \cite{DT20} for a survey, basic properties, and additional references). However we only need to use the following well known and basic facts (\cite[Lemma 2]{Ho73}): 
\begin{lemma}
Let $R$ be a normal equicharacteristic zero domain. Then $R$ is a splinter and hence $I^{+} = I$ for all ideals $I$ of $R$.
\end{lemma}
\begin{proof}
Let $S$ be a module finite extension of $R$. We can quotient $S$ by a minimal prime and assume it is a domain. Let $d$ be the degree of the extension of the corresponding fraction fields, i.e. $[Frac(S): Frac(R)]$. We have a trace map on the fraction fields, $Tr: Frac(S) \rightarrow Frac(R)$. Trace map preserves integral elements, that is for $s\in S$, $Tr(s) \in Frac(R)$ is integral over $R$ since $S$ is integral over $R$. Since $R$ is normal, it is integrally closed in $Frac(R)$ and hence $Tr(s) \in R$. One checks that $Tr(r) = dr$ for any $r \in R$ and hence $\frac{1}{d}Tr$ gives the desired splitting $S \rightarrow R$. Let $i\in I^{+}$, then by definition $i \in IR^{+}$ and we can write $i$ as a finite sum $i = \sum a_{j}r_{j}$ where $a_{j} \in I$ and $r_{j} \in R^{+}$. Let $S$ be a finite extension of $R$ containing all $r_{j}$, then we have $i \in IS$. Since $R$ is a splinter, we have $R \rightarrow S$ splits and since $i \in R$, we have $i\in I$.
\end{proof}

\par
We need to use the following standard property of coherent rings (\cite[Thm 2.3.2 6]{Gla89}) and recall the proof for the convenience of the reader:
\begin{lemma} \label{fingen}
Let $R$ be a coherent ring. Let $I$ be a finitely generated ideal of $R$ and let $a \in R$. Then the ideal $(I:a) = \{x | xa \in I \}$ is finitely generated.
\end{lemma}
\begin{proof}
Let $I = (u_{1},\cdots,u_{r})$ and let $F$ be the free module on generators $(x_{1},\cdots,x_{r+1})$. Consider the short exact sequence 
\begin{equation*}
    0 \rightarrow K \rightarrow F \rightarrow (I,a) \rightarrow 0
\end{equation*}
where the map from $F$ to $(I:a)$ is defined by $x_{i} \rightarrow u_{i}$ for $1\leq i \leq r$ and $x_{r+1} \rightarrow a$. Note that $K$ is finitely generated since $F$ is finitely generated and $(I,a)$ is finitely presented (as it is a finitely generated ideal of a coherent ring $R$). Let $F'$ be the free submodule of $F$ with generators $x_{1}, \cdots, x_{r}$. Any element of $K$ is of the form $u = a_{1}x_{1} + \cdots + a_{r+1}x_{r+1}$ where $a_{r+1} \in (I:a)$. Thus we have a map $K \rightarrow (I:a)$ defined by $u \rightarrow a_{r+1}$ and we also have the short exact sequence
\begin{equation*}
    0 \rightarrow K \cap F' \rightarrow K \rightarrow (I:a) \rightarrow 0
\end{equation*}
Hence $(I:a)$ is finitely generated.
\end{proof}

\section{Main theorems}
The following is our key ingredient:

\begin{citedthm}[{\citeleft\citen{RSS07}\citemid Example.\
  2.4\citeright}] \label{calc}
Let $k$ be a field of characteristic zero. Let $A = k[[x,y,z]]/(x^{3} + y^{3} + z^{3})$ or $A = k[x,y,z]/(x^{3} + y^{3} + z^{3})$. Then $z^{2}$ is inside the dagger closure of $(x, y)$ (with respect to any valuation $v$ of $A^{+}$ positive on $x$, $y$, and $z$ and $0$ on $k - \{0\}$). That is, there are elements of arbitrarily low valuation in the ideal $((x, y):_{A^{+}} (z^{2}))$.
\end{citedthm}

\begin{proof}
We wish to show that there are elements of arbitrarily low valuation in the ideal $((x, y):_{A^{+}} (z^{2}))$. Since  $ k[x,y,z]/(x^{3} + y^{3} + z^{3})$ is a subring of  $k[[x,y,z]]/(x^{3} + y^{3} + z^{3})$, the absolute integral closure of $k[x,y,z]/(x^{3} + y^{3} + z^{3})$ is a subring of the absolute integral closure of $k[[x,y,z]]/(x^{3} + y^{3} + z^{3})$. Hence we may work with $k[x,y,z]/(x^{3} + y^{3} + z^{3})$, as the elements we construct will be inside the absolute integral closure of $k[x,y,z]/(x^{3} + y^{3} + z^{3})$. We may assume that $k$ is algebraically closed and by a linear change of coordinates work with the ring $A = k[x,y,z]/(\theta x^{3} + \theta^{2}y^{3} + z^{3})$ where $\theta$ is a primitive third root of unity. The strategy is to embed $A$ into a copy of $A$ with `lower' grading via a `Frobenius-like' map and induction takes care of the rest. Since the embedding is constructed explicitly using equations with cube roots the elements will turn out to have arbitrarily low valuation with respect to \emph{any} valuation $v$ (positive on $x$, $y$, and $z$). Let $(x_{1}, y_{1}, z_{1})$ be such that
\begin{eqnarray*}
x_{1}^{3} = \theta^{\frac{1}{3}}x + \theta^{\frac{2}{3}}y \\
y_{1}^{3} = \theta^{\frac{1}{3}}x + \theta^{\frac{5}{3}}y \\
z_{1}^{3} = \theta^{\frac{1}{3}}x + \theta^{\frac{8}{3}}y
\end{eqnarray*}
\\
It follows from direct calculations that $A_{1} = k[x_{1}, y_{1}, z_{1}]$ is isomorphic to and contains $A$: $x$, $y$ are in $A_{1}$ as they are a linear combination of $(x_{1}, y_{1}, z_{1})$. The element $z$ is in $A_{1}$ as $(x_{1}y_{1}z_{1})^{3} = \theta x^{3} + \theta^{2}y^{3} = -z^{3}$ so $-z = x_{1}y_{1}z_{1}$, and $\theta x_{1}^{3} + \theta^{2}y_{1}^{3} + z_{1}^{3} = 0$; one checks that $(0 = \theta + \theta^{2} + 1)$ factors from the coefficients of each of $(x ,y, z)$ upon expanding. Define $A_{2}$ and $(x_{2}, y_{2}, z_{2})$ similarly. We have 
\begin{center}
    $x_{1}z^{2} = x_{1}x_{1}^{2}y_{1}^{2}z_{1}^{2} =  x_{1}^{3}y_{1}^{2}z_{1}^{2} =  (\theta^{\frac{1}{3}}x + \theta^{\frac{2}{3}}y)(y_{1}^{2}x_{1}^{2}) \in (x,y)$.
\end{center}
Also,
\begin{center}
    $x_{2}z^{2} = x_{2}x_{1}^{2}y_{1}^{2}z_{1}^{2} = x_{2}x_{1}^{2}y_{1}^{2}x_{2}^{2}y_{2}^{2}z_{2}^{2} = x_{2}^{3}x_{1}^{2}y_{1}^{2}y_{2}^{2}z_{2}^{2} = (\theta^{\frac{1}{3}}x_{1} + \theta^{\frac{2}{3}}y_{1})x_{1}^{2}y_{1}^{2}y_{2}^{2}z_{2}^{2} \in (x_{1}^{3}, y_{1}^{3}) \in (x,y)$.
\end{center}
Similarly we get that $x_{n}, y_{n}, z_{n}$ are in the ideal $((x,y): z^{2})$. 
\par
We check that the valuations become arbitrarily low as we iterate this process. Note that the formulae for $x_{n}$, $y_{n}$, $z_{n}$ imply that they have positive valuation. Suppose for the sake of contradiction there is a real number $\delta > 0$ such that $v(x_{n}), v(y_{n}), v(z_{n})$ is greater than or equal to $\delta$. Pick a large positive integer $N$ such that $(2N+1)\cdot \delta > v(z)$. Since $(x_{n}y_{n}z_{n})^{3} = -z_{n-1}^{3}$ for all $n$ by construction, we have $v(z_{N-1}) = v(x_{N}) + v(y_{N}) + v(z_{N}) > 3\delta$. Similarly, $v(z_{N-2}) = v(x_{N-1}) + v(y_{N-1}) + v(z_{N-1}) > \delta + \delta + 3\delta = 5\delta$. Repeating this $N$ times we get $v(z) = v(x_{1}) + v(y_{1}) + v(z_{1}) > (2N+1)\delta > v(z)$, a contradiction. Hence there are elements of arbitrarily low valuation in the ideal $((x,y):_{A^{+}} z^{2})$ for any valuation positive on $x$ and $y$.
\end{proof}

\begin{remark}
The above construction without the linear change of coordinates is given in \cite[Example 3.4]{Sin09}.
\end{remark}

\begin{remark}\label{mult}
A similar statement as Theorem \ref{calc} can be deduced using (algebro-)geometric ideas, see \cite[Theorem 3.4]{RSS07}. Since $A$ is (the completion of the cone over) an elliptic curve, there is a multiplication by $n$ endormorphism of rings of $A$ which has degree $n^{2}$ (in general morphisms of projective varieties do not correspond to morphisms of the corresponding graded rings, the statement here uses vanishing theorems of line bundles on abelian varieties). It follows that there are elements of degree $\frac{1}{n^{2}}$ in $((x, y): z^{2})$, and by iterating this map one gets that there are elements of arbitrarily low degree in the same ideal. It is a natural question whether one can conclude there are elements of arbitrarily low \emph{valuation} in this ideal for a valuation $v$ on $R^{+}$ as in Theorem \ref{calc} above using the multiplication map, however we do not know the answer \footnote{There seems to be a subtle change of language in literature in the context of the above dagger closure example while relating `low degree' and `low valuation' elements. \cite[last paragraph, Page 5]{RSS07} says that `` The examples are $\mathbb{N}$-graded, and in such cases it is natural to use the valuation arising from the grading: $v(r)$ is the least integer $n$ such that the $n$-th degree component of $r$ is nonzero" where as the subsequent work \cite[Page 10]{Sin09} mentions ``We focus next on an example; for graded domains, we use the grading in lieu of a valuation."}. A priori, a notion of a degree can only be extended to $R^{+GR}$ (refer to Section 3, \cite{RSS07}) and hence one would only get the non-coherence of $R^{+GR}$ this way.
\end{remark}

Here is our main theorem:
\begin{customthm}{A}\label{thm:Coho}
Let $R$ be an equicharacteric zero (i.e. $\mathbb{Q} \subset R$) noetherian complete local domain such that $\dim(R) \geq 2$. Then $R^{+}$ is not coherent.
\end{customthm}
\begin{proof}
Cohen Structure Theorem implies $R$ is module finite over a power series ring over a field. Hence we may assume $R$ itself is isomorphic to $k[[x_{1},\cdots,x_{{n}}]]$. The ideal $I = ((x_{1}, x_{2}):_{R^{+}} z^{2})$ where $z \in R^{+}$ is such that $z^{3} + x_{1}^{3} + x_{2}^{3} = 0$, is proper: if it contained $1$, we would have $z^{2} \in (x_{1}, x_{2})R^{+}$. However, $R[z]$ is normal and as normal rings are splinters (here we are using $R$ has equicharacteristic zero) we have that $z^{2} \in (x_{1}, x_{2})R[z]$. If $R$ is a ring then $R[[x]]$ is faithfully flat over $R$: it is flat over $R$ as it is flat over the polynomial ring $R[x]$ (which is free over $R$) and is faithful as there is a surjection $R[[x]] \rightarrow R$ sending $x$ to $0$. Hence $R[z]$ is faithfully flat over $k[[x_{1}, x_{2}]][z]$ and it follows that $z^{2} \in (x_{1}, x_{2})k[[x_{1}, x_{2}]][z]$ which is not true as for example the class [$\frac{z^{2}}{x_{1}x_{2}}$] spans the degree zero part of the top local cohomology of $k[[x_{1}, x_{2}]][z]$ (\cite[Example 2.4]{RSS07}). If $R^{+}$ were coherent this ideal would be finitely generated by Lemma \ref{fingen} (here we are using coherence). Consider the valuation arising in Lemma \ref{almo}, that is choose a extension of $R$ containing the generators of $I$ and choose a valuation dominating a prime ideal containing $I$ and all $x_{i}$. Then there are elements of arbitrarily low valuation in $I$ by Theorem \ref{calc}, which is a contradiction by Lemma \ref{almo}. 
\end{proof}

\begin{remark} \label{smi}
It is interesting to note why the above strategy fails if $k$ has positive characteristic. Indeed by standard tight closure computations, $z^{2}$ is inside the tight closure (which is the same as dagger closure in complete local domains by \cite{HH91}) of $(x_{1},x_{2})$ and hence one gets elements of arbitrarily low valuation in $((x_{1}, x_{2}):_{R^{+}}z^{2})$. However, we are crucially using the fact that normal rings are splinters in equicharacteristic zero. One cannot conclude in general that this ideal is proper and in fact it is not in this case: by a deep result of Smith \cite{Smi94}, tight closure is plus closure for ideals generated by systems of parameters (under mild assumptions) and hence $z^{2}$ will in fact be inside $(x_{1}, x_{2})R^{+}$. 
\end{remark}

Asgharzadeh's result follows immediately from (the proof of) Theorem \ref{thm:Coho} after localizing and we extend his results to dimension $3$, answering his question from \cite[Theorem 6.6]{Asg17}: 

\begin{corollary}\label{thm:corm}
Let $R$ be a noetherian complete local domain of mixed characteristic $(0,p)$. If $\dim(R) \geq 3$ then $R^{+}$ is not coherent.
\end{corollary}

\begin{proof}
Cohen Structure Theorem implies that $R$ is module finite over a power series ring over a DVR, and hence we may assume $R=V[[x_{1},\cdots, x_{n}]]$ where $V$ is a DVR. $R_{p} = R[\frac{1}{p}]$ need not be a complete local domain in equicharacteristic zero, however, $x_{1}, \cdots, x_{n}$ are still algebraically independent elements in $R_{p}$. Arbitrary localizations of coherent rings are coherent \cite[Thm 2.4.2]{Gla89} and absolute integral closures commute with localization, hence to show that $R^{+}$ is not coherent it is enough to show that $R^{+}[\frac{1}{p}] = R_{p}^{+}$ is not coherent. One similarly observes that there are arbitrarily small elements in $I = ((x_{1},x_{2}):_{R_{p}^{+}}z^{2})$; slight care is required here - $R_{p}$ may not be or contain a two dimensional \emph{power series} ring over a characteristic zero field, however it does contain a two dimensional polynomial ring over a characteristic zero field (as $x_{1}$ and $x_{2}$ are algebraically independent) and one can repeat the construction and calculations of Theorem \ref{calc} for the ring $V[\frac{1}{p}][x_{1}, x_{2}, z]/(x_{1}^{3} + x_{2}^{3} + z^{3})$. Choosing a valuation $v$ positive on $I$ as in the proof of Lemma \ref{almo}, we have that there are elements of arbitrarily low valuation in $I$ by Theorem \ref{calc}. This ideal is proper, as if $1 \in I$ then as $R_{p}[z]$ is normal it is a splinter and we have $z^{2} \in (x_{1},x_{2})R_{p}$. The completion of $R_{p}$ is $V[\frac{1}{p}][[x_{1}, \cdots, x_{n}, z]]/ (x_{1}^{3} + x_{2}^{3} + z^{3})$. It suffices to show the containment does not hold in the $(x_{1},\cdots, x_{n})$-adic completion as completion is faithfully flat, for which one can run the same argument as in Theorem \ref{thm:Coho} (which we can apply as $V[\frac{1}{p}]$ is a characteristic zero field). If $R_{p}^{+}$ was coherent we would have that the ideal $I$ is finitely generated which is not possible by Lemma \ref{almo}.
\end{proof}
In particular we do not use the direct summand theorem \cite{And18} or any of the recent advances in mixed characteristic commutative algebra (for example, \cite{Bha20}). Kazuma Shimomoto has given (in private communication) another proof of Asgharzadeh's theorem \cite[Theorem 6.6]{Asg17} using the main results of \cite{Bha20} which we reproduce here, with permission, for the sake of completeness. Similar to Asgharzadeh's methods it uses highly non-trivial results and is unable to achieve our $\dim(R)=3$ improvement. We note that Shimomoto's proof allows us to relax the assumptions of Asgharzadeh's theorem from a complete local domain to an excellent local domain, since the corresponding statements from \cite{Bha20} and \cite{BMP+20} have (the weaker) excellence assumptions. The author thanks Bhargav Bhatt for pointing this out.
\\
\begin{theorem} \label{asgthm}
  Let $R$ be an excellent local domain of mixed characteristic $(0,p)$. If $\dim(R) \geq 4$ then $R^{+}$ is not coherent.
\end{theorem}
\begin{proof}
(Shimomoto) If $R$ is an excellent local domain the $p$-adic completion of $R^{+}$, $\widehat{R^{+}}$ is a big balanced Cohen-Macaulay algebra by \cite[Corollary 2.9]{BMP+20}. In general completions of non-noetherian rings may not be (faithfully) flat over the ring, however, if $R^{+}$ is coherent, completions with respect to finitely generated ideal are flat by \cite[Theorem 3.58]{Has11}. Hence $\widehat{R^{+}}$ is flat over $R^{+}$. Since $R$ is local $p$ is in all maximal ideals of $R^{+}$ (by the Cohen-Seidenberg theorems) and $\widehat{R^{+}}/p \simeq R^{+}/p$ it follows that $\widehat{R^{+}}$ is faithfully flat over $R^{+}$. This implies $R^{+}$ is big balanced Cohen-Macaulay, which is not true if $\dim(R)\geq 4$ by \cite[Proposition 3.6]{AH97} or \cite[Proposition 2.1]{ST21} (the contradiction is obtained by a standard trick of localizing at $p$ and using that $R^{+}$ is not Cohen-Macaulay if $\dim(R)\geq 3$ when $R$ has characteristic zero). This strategy fails when $dim(R) = 3$ because $R^{+}$ itself \emph{is} Cohen-Macaulay in this situation by \cite[Corollary 2.3]{ST21}.
\end{proof}
\par

\begin{remark}
 One can observe that the non-coherence of absolute integral closures (in characteristic zero) relies on the non-triviality of dagger closure and the fact that normal rings (even non-noetherian ones) are splinters. Our complete local domain assumption allows us to reduce to the case of the explicit Theorem \ref{calc}, where we know the dagger closure is non-trivial. Hence it seems that the complete local domain assumption should not be essential in the statements of our main results (Theorem \ref{thm:Coho}, Corollary \ref{thm:corm}).
\end{remark}

\begin{remark} \label{coherence}
Note that when $\dim(R) = 1$, $R$ is regular if it is normal and finite domain extensions of normal dimension $1$ domains are flat. Hence $R^{+}$ \emph{is} coherent as it is a (flat) colimit of normal domains (a flat direct limit of coherent rings is coherent, \cite[Theorem 2.3.3]{Gla89}). The interested reader can observe that this also follows from the fact that $R^{+}$ is a Pr\"{u}fer domain. When $R$ has positive characteristic, Aberbach and Hochster show that $R^{+}$ is not coherent for any noetherian domain $R$ if $\dim(R) \geq 3$ and if $R$ is a complete local domain containing a field of positive transcendence degree over a finite field if $\dim(R) = 2$. It is not known whether $R^{+}$ (where $R$ is a complete local domain) is coherent in positive and mixed characteristics when $\dim(R) = 2$ in complete generality. It is noteworthy that we are able to settle the coherence problem in all dimensions (for complete local domains) in characteristic zero.  
\end{remark}

\section{Coherence in mixed characteristic}

Let $R$ be a domain of mixed characteristic $(0,p)$. Throughout this section we will denote the $p$-adic completion of $R^{+}$ by $\widehat{R^{+}}$. In mixed characteristic \cite[Corollary 5.17]{Bha20} says that $\widehat{R^{+}}$ is Cohen-Macaulay where $R$ is a complete local domain of mixed characteristic (where as $R^{+}$ is merely `cohomologically Cohen-Macaulay', i.e. $H^{i}_{m}(R^{+}) = 0$ for $0\leq i \leq \dim(R)-1$, $m$ the maximal ideal of $R^{+}$, and $H^{i}_{m}(R^{+})$ is the local cohomology, see \cite[Theorem 5.1]{Bha20}). For clarification on Bhatt's result and the precise definition(s) of Cohen-Macaulay, see \cite[Discussion 2.2]{BMP+20}. Our methods do not show whether $\widehat{R^{+}}$ is coherent. 

\begin{question}\label{que1}
Let $R$ be a noetherian complete local domain of mixed characteristic. Is the $p$-adic completion of $R^{+}$, $\widehat{R^{+}}$, coherent?
\end{question}

In the positive characteristic version of our results, Aberbach and Hochster first use the homological properties of perfect rings to (implicitly) show that if $\widehat{R^{+}}$ is coherent then it is a G.C.D domain, see proof of \cite[Proposition 4.4]{AH97} and ``Thus $a$ and $b$ have a G.C.D in $R_{s}^{+}$''. Elementary properties of divisor class groups then imply that every complete local domain has a torsion divisor class group, which is not true if dimension of the ring is $3$ or more. Using the corresponding homological properties of perfectoid rings, we are able to obtain the following partial result in mixed characteristic by mirroring their methods:
\begin{theorem}\label{mixed}
Let $R$ be a complete noetherian local domain of mixed characteristic. If $\widehat{R^{+}}$ is coherent then it is a G.C.D domain.
\end{theorem}
\begin{proof}
Note that $\widehat{R^{+}}$ is a perfectoid ring \cite[Example 3.8 (2)]{BIM18}. As in the proof of Theorem \ref{thm:corm} we may assume $R$ is a power series ring over a discrete valuation ring $V[[x_{1},\cdots,x_{n-1}]]$ and hence $p,x_{1},\cdots, x_{n-1}$ form a s.o.p for $R$. Consider the radical ideal $m = \sqrt{p,x_{1},\cdots, x_{n-1}}$ where the radical is taken in $\widehat{R^{+}}$ and hence $m$ is an ideal of $\widehat{R^{+}}$. The ideal $m\widehat{R^{+}}/\sqrt{p\widehat{R^{+}}}$ of $\widehat{R^{+}}/\sqrt{p\widehat{R^{+}}}$ is generated up to radical by $n$ elements, hence by \cite[Lemma 3.7]{BIM18} flat dimension of $\widehat{R^{+}}/m$ over $\widehat{R^{+}}$ is less than or equal to $n$\footnote{The proof of this for $p$-torsion free perfectoid rings follows immediately from the corresponding properties of perfect rings.}. Note that $m$ is the maximal ideal of $\widehat{R^{+}}$: since $p\in m$, $\widehat{R^{+}}/m = R^{+}/\sqrt{p,x_{1},\cdots, x_{n-1}}$ is a field (since $R^{+}$ is an integral extension of $R$). 
\par
Let $I$ be a finitely generated ideal of $\widehat{R^{+}}$. Since $\widehat{R^{+}}$ is coherent, $\widehat{R^{+}}/I\widehat{R^{+}}$ has a resolution by finitely generated $\widehat{R^{+}}$ modules by \cite[Corollary 3.5.2]{Gla89}. Since $\widehat{R^{+}}$ is local there is a minimal resolution of $\widehat{R^{+}}/I\widehat{R^{+}}$ and the $i$th betti number is given by the vector space dimension over $\widehat{R^{+}}/m\widehat{R^{+}}$ of $Tor_{i}^{\widehat{R^{+}}}(\widehat{R^{+}}/I\widehat{R^{+}}, \widehat{R^{+}}/m)$. Since the flat dimension of $\widehat{R^{+}}/m$ is less than or equal to $n$, we have that the projective dimension of $\widehat{R^{+}}/I\widehat{R^{+}}$ over $\widehat{R^{+}}$ is less than or equal to $n$. 
\par
The above discussion implies that the projective dimension of finitely generated ideals of $\widehat{R^{+}}$ is finite. Hence $\widehat{R^{+}}$ is by definition a (non-noetherian) regular ring \cite[Definition on page 200]{Gla89}. Since $R$ is complete and local, $R^{+}$ and consequently $R^{+}/p$ are local. Hence $\widehat{R^{+}}$ is a local ring. Coherent regular local rings are G.C.D domains by \cite[Corollary 6.2.10]{Gla89}. 
\end{proof}

\begin{remark}
In the proof above we show that if $\widehat{R^{+}}$ is coherent then it is also regular and that a coherent regular local ring is a G.C.D domain (just as in the case of noetherian regular local rings) and in particular that it is an integral domain. This assertion is not obvious - the main result of Heitmann's recent work says that if $R$ is a complete local domain then $\widehat{R^{+}}$ is (surprisingly) an integral domain \emph{regardless} of our coherence (and regular) hypothesis \cite[Theorem 1.7]{Hei21}.
\end{remark}

\begin{remark}\label{val}
 As mentioned in Remark \ref{coherence}, if $R$ is regular, dimension $1$, and of positive characteristic, then $R^{+}$ is coherent and it follows from the proof of the above Theorem \ref{mixed} that projective dimension of $R^{+}/IR^{+}$ (over $R^{+}$) is $\leq1$ if $I$ is a finitely generated ideal since $R^{+}$ is a perfect ring. This may fail if the ideal is not finitely generated: if $R$ is complete and of dimension $1$ then the global dimension of $R^{+}$ is the projective dimension of $R^{+}/m$ and is equal to $2$ where $m$ is the maximal ideal of $R^{+}$ (note that it is not finitely generated). In this situation we may assume $R$ to be a complete discrete valuation ring and hence $R^{+}$ is a valuation ring (regardless of the characteristic of $R$, see proof of Lemma \ref{almo}). The result now follows from the theory of valuation rings - the global dimension of valuation rings was studied by Osofsky \cite{Oso67} and a more modern and simplified treatment sufficient for the preceding result can be found in the works of Asgharzadeh \cite[Theorem 1.1]{Asg10} and Datta \cite[Theorem 4.2.5]{Dat17}. 
 \end{remark}
 
 \begin{remark}
 In Remark \ref{val} we used the fact that absolute integral closures of complete (or more generally, henselian) valuation rings are valuation rings. This is quite useful in studying absolute integral closures in dimension $1$, or more generally, valuation rings. For example, $2$ out of $3$ proofs of `Valuation rings are derived splinters' by Antieau and Datta \cite{AD21} use the preceding fact to reduce to the case of absolutely integrally closed valuation rings. As another example, although by Remark \ref{val} the projective dimension of $R^{+}/m$ is $2$ (where $dim(R) = 1$, $R$ is complete, and $m$ is the maximal ideal of $R^{+}$), the $Tor$ dimension of $R^{+}/m$ (over $R^{+}$) is $1$ as we have a short exact sequence $0 \rightarrow m \rightarrow R^{+} \rightarrow R^{+}/m \rightarrow 0$ and $m$ is a flat because it is a torsion free module over the valuation ring $R^{+}$. The interested reader can observe that this answers \cite[Question 3.7]{AH97} when $dim(R) = 1$.  
\end{remark}

We expect $\widehat{R^{+}}$ to be not coherent as well. Unfortunately we do not know how to obtain a contradiction assuming $\widehat{R^{+}}$ is coherent given Theorem \ref{mixed}. One cannot use the same properties of divisor class groups as in \cite{AH97} as $\widehat{R^{+}}$ is not an integral extension of $R$. Motivated by the statements above, we would like to include, with permission, the following related question communicated to us by Kazuma Shimomoto:
\begin{question} \label{que2}
(Shimomoto) Does every noetherian local domain map to a coherent big Cohen-Macaulay algebra?
\end{question}

When $R$ is excellent of positive characteristic $R^{+}$ is a big Cohen-Macaulay algebra by \cite{HH94} and not coherent by \cite{AH97} hence one might suspect that the question has a negative answer. However, the answer is yes for a class of positive characteristic rings called `$F$-coherent' rings:
\begin{citeddef}[{\citeleft\citen{Shi10}\citemid Def.\ 3.2\citeright}]
  Let $R$ be a ring of positive characteristic. $R$ is said to be `$F$-coherent' if the perfect closure of $R$, $R_{perf}$, is a coherent ring.
\end{citeddef}
It is well known that under mild assumptions\footnote{such as $R$ is a quotient of a Gorenstein ring, see \cite[Theorem 3.11]{Shi10}} on the ring $R$, $R_{perf}$ is an `almost Cohen-Macaulay' algebra in positive characteristic and if it is coherent, it follows (refer to \cite[Theorem 3.11]{Shi10}) that it is Cohen-Macaulay, coherence does not allow flexibility of almost mathematics as indicated by the proofs of our main theorems above (for a more precise version of this statement see \cite[Theorem 3.12]{AS12} which roughly states that `coherence and almost Cohen-Macaulay implies (big balanced) Cohen-Macaulay'). Hence the answer to the above question is yes if the domain is $F$-coherent. It follows from Kunz's theorem that regular rings are $F$-coherent (\cite[Proposition 3.3]{Shi10}) and hence so is any purely inseparable extension of a regular ring (as both have the same perfection). The question of whether these are all $F$-coherent rings is one of the key open problems in the theory (\cite[Question 4]{Shi10}). 

\begin{remark}
If $S$ is a module finite purely inseparable extension of a regular $F$-finite ring $R$ of positive characteristic, $S$ is contained inside $R^{\frac{1}{p^{e}}}$ for $e$ large enough (take maximum of all $e$ such that the generators of $S$ as an $R$-module are contained in $R^{\frac{1}{p^{e}}}$). $R^{\frac{1}{p^{e}}}$ is isomorphic to $R$ as an algebra and hence $S$ not only has a coherent Cohen-Macaulay algebra, it has a \emph{noetherian} Cohen-Macaulay (and in fact, regular) algebra . 
\end{remark}

\begin{remark}
The proof of Theorem \ref{mixed} offers insights into the characterization of $F$-coherent rings. The proof and homological properties of perfect rings show that if $R$ is an excellent noetherian normal $F$-coherent domain then $R_{perf}$ is a G.C.D domain (see also \cite[Corollary 3.7]{Asg17}). Let $(a,b)$ be a height one ideal of $R$. As $R_{perf}$ is a G.C.D domain $a,b$ have a G.C.D $d$ in $R_{perf}$. Let $S$ be the normalization of $R[d]$, $S$ is purely inseparable over $R$ (in general normalizations of $F$-coherent rings are $F$-coherent by \cite[Theorem 3.8]{Shi10}). Standard properties of divisor class groups (see discussion after \cite[Lemma 4.3]{AH97}) imply that there are maps $Cl(R) \rightarrow Cl(S)$ and $Cl(S)\rightarrow Cl(R)$ with the latter map multiplication by $[S:R]$. Since $(a,b)$ is trivial in the divisor class group of $S$ and $S$ is purely inseparable over $R$, it follows that $(a,b)$ is $p^{\infty}$ torsion in $Cl(R)$. Hence $Cl(R)$ is $p^{\infty}$ torsion for any excellent normal noetherian $F$-coherent domain $R$. The preceding discussion shows that the divisor class group of any $F$-coherent ring and any purely inseparable extensions of it (which is also $F$-coherent) is $p^{\infty}$ torsion under mild conditions on the ring. We do not know if it is reasonable to expect a converse. Under similar assumptions, if every purely inseparable extension of a ring $R$ has a $p^{\infty}$ torsion divisor class group, or if the perfection of the ring is a G.C.D domain, is the ring $F$-coherent? $p^{\infty}$ torsion in divisor class groups has historically often shown up as a separate case which requires new methods to deal with in problems in $F$-singularity theory and algebraic geometry. See for example \cite{PS20} and \cite{Kol21}. 
\end{remark}

\section{An analogue of Smith's result}

\par 
As remarked in Remark \ref{smi}, our methods are closely related to the deep result of Smith (\cite{Smi94}), hence we would like to ask if the following analogous result holds in mixed characteristic:
\begin{question}\label{smiana}
Let $I$ be a parameter ideal of a mixed characteristic complete local domain $R$ and fix a valuation $v$ on $R^{+}$ positive on $I$. Does the dagger closure of $I$ (with respect to $v$) coincide with $I\widehat{R^{+}}\cap R$?
\end{question}

We illustrate the question in a series of propositions below using the running example in our note, $R = \mathbb{Z}_{p}[[x, y, z]]/(x^{3} + y^{3} + z^{3})$. First we explain why $z^{2}$ is not in the plus closure of $I=(x,y)$ (i.e. $IR^{+}\cap R$) and that $z^{2}$ is in fact in the dagger closure of $I$. This is implicit in the proofs of our main theorems in Section $3$ and explains why we ask for the ``$\widehat{R^{+}}$-closure of $I$'' (i.e. $I\widehat{R^{+}}\cap R$)  to coincide with the dagger closure instead of the usual plus closure of $I$. Then we show that $z^{2}$ is in $(p^{n}, x, y)R^{+}$ for every $n$ by exploiting the fact that $R$ is the (completion of) the cone over an abelian variety - the $p^{n}$ power maps on the abelian variety lift to ring endomorphisms of $R$. Then we use \cite{Bha20} to show that $z^{2}$ is in $I\widehat{R^{+}}$.

\begin{proposition}
Let $R = \mathbb{Z}_{p}[[x, y, z]]/(x^{3} + y^{3} + z^{3})$ and $I = (x,y)$, then $z^{2}$ is not in the plus closure of $I$.
\end{proposition}

\begin{proof}
If $z^{2}$ was in the plus closure of $I$, as absolute integral closures commute with localization, after localizing at $p$ we have that $z^{2} \in (x,y)R[\frac{1}{p}]^{+}$. This implies $z^{2} \in (x,y)R[\frac{1}{p}]$ as $R[\frac{1}{p}]$ is a normal equi-characteristic zero domain and hence a splinter. Finally, $z^{2} \notin (x,y)R[\frac{1}{p}]$ as the containment can be checked after $(x,y)$-adic completion (which is faithfully flat) and the $I$-adic completion of $R[\frac{1}{p}]$ is $\mathbb{Q}_{p} [[x,y,z]]/(x^{3}+y^{3}+z^{3})$ where the containment does not hold.
\end{proof}

\begin{proposition}
Let $R = \mathbb{Z}_{p}[[x, y, z]]/(x^{3} + y^{3} + z^{3})$ and $I = (x,y)$, then $z^{2}$ is in the dagger closure of $I$.
\end{proposition}

\begin{proof}
The reader can observe that the same argument as in Theorem \ref{calc} works: the absolute integral closure of polynomial ring $\mathbb{Z}_{p}[x,y,z]/(x^{3}+y^{3}+z^{3})$ is inside the absolute integral closure of $R$, hence we may assume $R = \mathbb{Z}_{p}[x,y,z]/(x^{3}+y^{3}+z^{3})$. We may also enlarge $R$ and assume it contains $\theta^{\frac{1}{3}}$ where $\theta$ is the primitive third root of unity. The construction in Theorem \ref{calc} give a finite extension of $R$ isomorphic to itself. In the new extension $((x,y):z^{2})$ contains elements of lower valuation and by iterating this process the calculation performed at the end of Theorem \ref{calc} shows that there are elements of arbitrary low valuation in the colon ideal $((x,y):_{R^{+}}z^{2})$. 
\end{proof}

\begin{proposition}\label{powermap}
Let $R = \mathbb{Z}_{p}[[x, y, z]]/(x^{3} + y^{3} + z^{3})$, then $z^{2}$ is in $(p^{n}, x, y)R^{+}$ for all $n \geq 1$.
\end{proposition}

\begin{proof}
Let $S = \mathbb{Z}_{p}[x, y, z]/(x^{3} + y^{3} + z^{3})$. Note that $S$ is the homogeneous coordinate ring of an elliptic curve $Proj(S) = E$ over $\mathbb{Z}_{p}$. In general morphisms of projective varieties do not correspond to morphisms of the corresponding graded rings, however for abelian varieties, the $p^{n}$ power maps $[p^{n}]$ lift to module-finite graded \emph{ring} endomorphisms $m_{p^{n}}$ of the homogenous coordinate ring $S$, see discussion after \cite[(3.4.1)]{RSS07}, paragraph starting with ``Thus we have a map of graded rings from $R$ to itself...". The maps $[p^{n}]:E \rightarrow E$ induce multiplication by $p^{n}$ maps on $H^{1}(E, \mathcal{O}_{E}) \simeq \mathbb{Z}_{p}$, this is a standard fact about abelian schemes and we refer the reader to \cite[Example 4.14]{BMP+20} for a proof. Since the degree zero part of $H^{2}_{(x,y,z)}(S)$ is isomorphic to $H^{1}(E, \mathcal{O}_{E})$ and by construction of $m_{p^{n}}$ is defined by the same \v{C}ech complex, $m_{p^{n}}: (H^{2}_{(x,y,z)}(S))_{0} \rightarrow (H^{2}_{(x,y,z)}(S))_{0}$ is also the multiplication by $p^{n}$ map.
\par
Note that $\sqrt{(x,y)} = \sqrt{(m_{p^{n}}(x), m_{p^{n}}(y))} = (x,y,z)$ is the irrelevant ideal. The class $[\frac{z^{2}}{xy}]$ spans $(H^{2}_{(x,y,z)}(S))_{0}$ and let $[\frac{\alpha}{m_{p^{n}}(x)m_{p^{n}}(y)}]$ be another generator of $(H^{2}_{(x,y,z)}(S))_{0}$ obtained by computing the \v{C}ech complex with respect to $(m_{p^{n}}(x), m_{p^{n}}(y))$. Hence we have $[\frac{m_{p^{n}}(z^{2})}{m_{p^{n}}(x)m_{p^{n}}(y)}] = [p^{n}\frac{\alpha}{m_{p^{n}}(x)m_{p^{n}}(y)}]$ as classes in $(H^{2}_{(x,y,z)}(S))_{0}$. Unwinding the definition of $H^{2}_{(x,y,z)}(S)$ using the \v{C}ech complex we have $\frac{m_{p^{n}}(z^{2})}{m_{p^{n}}(x)m_{p^{n}}(y)} = p^{n}\frac{\alpha}{m_{p^{n}}(x)m_{p^{n}}(y)} + \frac{\beta_{1}}{m_{p^{n}}(y)} + \frac{\beta_{2}}{m_{p^{n}}(x)}$ for some $\beta_{1}, \beta_{2} \in S$ and cross multiplying by $m_{p^{n}}(x)m_{p^{n}}(y)$ we have $m_{p^{n}}(z^{2})\in(p^{n}, m_{p^{n}}(x) , m_{p^{n}}(y))$ for all $n$. Since $m_{p^{n}}: S \rightarrow S$ is a module finite morphism, we have a copy of $S$, $T$ in $S^{+}$ such that the inclusion $S \rightarrow T$ is the same as $m_{p^{n}}: S \rightarrow S$. Hence $z^{2}\in(p^{n}, x, y)S^{+}$. Since $S$ is contained in $R$ we have $S^{+} \subset R^{+}$ and hence $z^{2}\in(p^{n}, x, y)R^{+}$.  
\end{proof}

The anonymous referee kindly pointed out to us that this proposition also follows from the fact that since $p$ is a non-zero divisor on $R^{+}/(x,y)$ by \cite{Bha20} the derived $p$ adic completion of $R^{+}/(x,y)$ coincides with itself and hence is in particular $p$-adically separated.
\\
\begin{proposition} \label{imp}
Let $(R,m)$ be an excellent local domain of mixed characteristic $(0,p)$ and let $(p, x, y)$ denote a system of parameters. Then $\cap_{n}(p^{n}, x, y)R^{+} \subset (x, y)\widehat{R^{+}}$. 
\end{proposition}
\begin{proof}
We will be using the main theorems of \cite{Bha20}, that is by \cite[Corollary 5.11]{Bha20} we know that $(p^{n}, x, y)$ is a regular sequence on $R^{+}$ for every $n$. The reader is encouraged to look at \cite[Lemma 3.2]{HM21} where this proposition is proved for a two element ideal $(p,g)$ and does not use \cite{Bha20}, since Cohen-Macaulayness of $R^{+}$ in dimension $2$ follows from elementary commutative algebra. Let $\alpha \in \cap_{n}(p^{n}, x, y)R^{+}$. We have $\alpha = a_{1}x + b_{1}y + c_{1}p$ for some $(a_{1}, b_{1}, c_{1}) \in R^{+}$ and we also have that $c_{1}p \in \cap_{n}(p^{n}, x, y)$. So $c_{1}p = a_{2}x + b_{2}y + c_{2}p^{2}$ for some $(a_{2}, b_{2}, c_{2}) \in R^{+}$. We will now use that $(p, x, y)$ is a regular sequence to show that we may choose $(a_{2}, b_{2})$ to be divisible by $p$. Since $c_{1}p = a_{2}x + b_{2}y + c_{2}p^{2}$, we have that $(a_{2}x + b_{2}y) = 0$ as an element in $R^{+}/p$. Hence $b_{2}y = 0$ as an element of $R^{+}/(p, x)$ and as $(p, x, y)$ form a regular sequence on $R^{+}$ we have $b_{2} = gx + hp$ for some $(g,h) \in R^{+}$ and hence $c_{1}p = (a_{2}+gy)x + hpy + c_{2}p^{2}$. Since $x$ is a non-zero divisor on $R^{+}/p$ we have $(a_{2} + gy) = jp$ for some $j \in R^{+}$ and hence $\alpha = a_{1}x + b_{1}y + (a_{2}+gy)x + hpy + c_{2}p^{2}$. Thus we may choose $a_{2}$ and $b_{2}$ to be divisible by $p$. We repeat this process to get $(a_{i}, b_{i})$: we have $c_{2}p^{2} \in \cap_{n}(p^{n}, x, y)$ and hence we may write $c_{2}p^{2} = a_{3}x + b_{3}y + c_{3}p^{3}$ and by the same argument with $p$ replaced by $p^{2}$ we may assume $a_{3}, b_{3}$ to be divisible by $p^{2}$. \\
Continuing this way we have that $(a_{i}, b_{i}) \in p^{i-1}R^{+}$ and hence the sums $(a_{1}+a_{2}+...), (b_{1}+b_{2}+...)$ converge and are well defined elements of $\widehat{R^{+}}$. Finally, we have that $\alpha = a_{1}x + b_{1}y + a_{2}x + b_{2}y +... = x(a_{1}+a_{2}+...) + y(b_{1}+b_{2}+...) \in (x,y)\widehat{R^{+}}$.  
\end{proof}

\begin{corollary}
Let $R = \mathbb{Z}_{p}[[x, y, z]]/(x^{3} + y^{3} + z^{3})$ and $I = (x,y)$. Then $z^{2}$ is not in $I^{+} = IR^{+}\cap R$ but is in the dagger closure of $I$, $I^{\dagger}$, and the ``$\widehat{R^{+}}$-closure of $I$'', $I\widehat{R^{+}}\cap R$.
\end{corollary}
\begin{proof}
Follows from the propositions above.  
\end{proof}

\begin{remark}
Note that valuations on $R$ (positive on $I$) extend to those on $\widehat{R^{+}}$ by \cite[Lemma 1.3]{Hei21}. Given this, one can quickly conclude that a negative answer to Question \ref{smiana} for any finitely generated ideal $I$ of $R^{+}$ (not just one generated by system of parameters) implies that the answer to Question \ref{que1} is negative, that is $\widehat{R^{+}}$ is not coherent (as any element of $I\widehat{R^{+}}$ will have valuation at least the minimum of valuations of generators of $I$). This is analogous to the fact that if absolute integral closures were coherent in positive characteristic then tight closure would be equal to plus closure and hence would commute with localization. Brenner and Monsky's paper \cite{BM10} which settled the long standing question of whether tight closure localizes (in the negative) immediately implies that $R^{+}$ is not coherent in positive characteristic for their choice of $R$ (see \cite[Proposition 4.5]{Shi10} and \cite[Proposition 2.3]{AH97}).
\end{remark}

\section{Acknowledgements}
The author thanks Mohsen Asgharzadeh, Bhargav Bhatt, Rankeya Datta, Arnab Kundu, Vaibhav Pandey, Kazuma Shimomoto, and his advisor Kevin Tucker for conversations, mentorship, encouragement, useful suggestions, or a reading of the manuscript, and especially Kazuma Shimomoto for the permission to include Question \ref{que2} and his proof of Theorem \ref{asgthm}. The anonymous referee suggested improvements in exposition and pointed out mistakes and serious gaps in a former version of this article (especially in what is now Section $5$) which we are grateful for.

\end{document}